\documentclass{amsart}
\usepackage{amssymb}
\usepackage{amscd}
\usepackage{color}

\newcommand{\Diff}{\mathcal{D}}
\newcommand{\Emb}{\mathcal{E}}

\newcommand{\id}{\text{id}}

\newcommand{\Laplacian}{\Delta}
\DeclareMathOperator{\grad}{grad}
\DeclareMathOperator{\ad}{ad} 
\DeclareMathOperator{\sgrad}{sgrad}
\DeclareMathOperator{\diver}{div} \DeclareMathOperator{\curl}{curl}

\DeclareMathOperator{\Exp}{Exp}
\newcommand{\llangle}{\langle\!\langle}
\newcommand{\rrangle}{\rangle\!\rangle}

\newcommand{\SemiNM}{C^{\infty}(N,M\times \mathbb{R})}

\DeclareMathOperator{\Jac}{Jac}
\newtheorem{theorem}{Theorem}[section]

\newtheorem{lemma}[theorem]{Lemma}
\newtheorem{corollary}[theorem]{Corollary}

\theoremstyle{definition}

\newtheorem{exmp}[theorem]{Example}

\newtheorem{remark}[theorem]{Remark}

\theoremstyle{remark}

\begin{document}

\title{The geometry of barotropic flow}
\author{Stephen C. Preston}
\address{Department of Mathematics, University of Colorado, Boulder,
CO 80309-0395} \email{Stephen.Preston@colorado.edu}

\maketitle

\tableofcontents

\section{Introduction}

In this article we write the equations of barotropic compressible fluid mechanics as a geodesic equation on an infinite-dimensional manifold. 
The equations are given by
\begin{align}
u_t + \nabla_uu = -\frac{1}{\rho} \grad p, \label{evolution} \\
\rho_t + \diver{(\rho u)} = 0, \label{continuity}
\end{align}
where the fluid fills up a compact manifold $M$, $u$ is a time-dependent velocity field on $M$, and $\rho$ is the density, a positive function on $M$. The barotropic assumption is that the pressure $p$ is some given function of the density, although our methods also extend to certain more general isentropic flows.
Our infinite-dimensional manifold is the product $\Diff(M)\times C^{\infty}(M,\mathbb{R})$. This is a group using the semidirect product (which is sometimes incorporated in other treatments), but the Riemannian metric we use is neither left- nor right-invariant. Hence our geodesic equation is \emph{not} an Euler-Arnold equation. We compute the sectional curvature and show that at least when $M=S^1$, the curvature is always nonnegative. We also establish some results on the Lagrangian linear stability of solutions of this system, for certain nonsteady solutions in one dimension  and steady solutions in two dimensions.

It has been known for many years that the system \eqref{evolution}--\eqref{continuity} can be derived via Hamilton's principle on an infinite-dimensional manifold (see e.g., Ebin~\cite{ebin}, Smolentsev~\cite{smolentsev}, and Holm et al.~\cite{HMR}, but all such approaches use a nonzero potential energy, so that the equations are essentially Newton's equation $\frac{D}{dt} \frac{d\eta}{dt} = -\grad \Psi$ for some nonzero $\Psi$ rather than a pure geodesic motion (with $\Psi=0$). We review this approach in Section \ref{backgroundsection}. In contrast we obtain a genuine geodesic equation, but the price we pay is that our geodesic equation only gives the barotropic equations on a certain nonholonomic distribution (i.e., some geodesics correspond to barotropic fluids, while others have no physical meaning). Our method is somewhat closer to the metric introduced in Eisenhart~\cite{eisenhart} on the product $\Diff(M)\times \mathbb{R}^2$, where extra degrees of freedom are introduced and weighted in the Riemannian metric by terms involving the potential energy and an arbitrary parameter. The difference is that our metric is ``diagonal'' on the product, and the extra variables we introduce are directly related to the density.

One motivation for doing this is to obtain some intuition for the behavior of compressible fluids under a perturbation. In much the same way that a one-dimensional particle trajectory in a convex potential energy behaves approximately like a geodesic in a surface of positive curvature, we hope to understand stability of a compressible fluid using curvature computations which at least intuitively suggest the behavior of perturbed solutions. Another motivation is to understand the warped product geometry of $\Diff(M) \times C^{\infty}(M,\mathbb{R})$, under a noninvariant metric. Invariant metrics have some nice algebraic properties, but are often not physically relevant (for example, the right-invariant $L^2$ metric on $\Diff(M)$ does not come from any kind of physics, and its geodesic equation has no known physical relevance; while the non-invariant $L^2$ metric we consider is naturally related to the kinetic energy, and its geodesic equation describes the motion of a force-free family of particles).

In Section \ref{semidirectgeometry}, we write down the Riemannian metric and its geodesics along with its Riemannian curvature. The advantage of the geodesic approach, as pioneered by Arnold~\cite{arnold} for incompressible fluids, is that one can in principle use the curvature to discuss Lagrangian stability of the fluid: intuitively, if the sectional curvature is positive, then the particle paths should be stable under small perturbations of the initial velocity. (More precisely one uses the Rauch theorem which gives results up to the first conjugate point.) This cannot be applied rigorously in our case, since unlike in the case of incompressible fluids~\cite{ebinmarsden}, the Riemannian exponential map cannot be smooth in any Hilbert topology: simple examples show that geodesics  in the present context are not even locally minimizing. Hence our approach is necessarily only formal, and so we may as well require all objects to be $C^{\infty}$ rather than working in Sobolev spaces.

Nonetheless we can still analyze the Jacobi equation for linear perturbations of geodesics directly, which we do in Section \ref{jacobisection}; see \cite{BG} for a similar perspective. We establish stability or weak instability in special cases (for arbitrary solutions in one space dimension and a rigid rotational flow in two space dimensions). This is not a significant drawback, since even for incompressible fluids a direct approach to stability is frequently more informative than an analysis of the curvature (see \cite{prestonstability}). In the one-dimensional case we show that Lagrangian perturbations grow at most linearly in time up to the shock (when solutions cease to be smooth and our methods no longer apply), while in the two-dimensional steady case, many Lagrangian perturbations are bounded for all time. This portion of the paper previously appeared in the author's thesis~\cite{thesis}.

\section{Background}\label{backgroundsection}

We describe a compressible fluid as follows: consider a manifold $N$ with a volume form $\nu$ (describing the mass distribution) and a Riemannian manifold $M$ with metric $\langle \cdot, \cdot\rangle$ and Riemannian volume form $\mu$ (describing physical space). Fluid configurations are described by trajectories $\eta(t)\in C^{\infty}(N,M)$, where the density is defined by 
\begin{equation}\label{densitydef}
\rho\circ\eta \eta^*\mu =  \nu, \quad \text{or } \Jac(\eta) \rho\circ\eta.
\end{equation} 
This space is formally\footnote{To prove theorems about geometry on infinite-dimensional manifolds rigorously, one prefers a Banach or Hilbert structure; however those theorems will not apply in this case regardless, so we lose nothing by staying in the $C^{\infty}$ category.} a manifold, where the tangent spaces are 
$$T_{\eta}C^{\infty}(N,M) = \{U\colon N\to TM \, \vert \, U(p)\in T_{\eta(p)}M \,\forall p\in N\}$$
and the coordinate charts are given by the exponential maps 
$$ \Exp_{\eta}(U) = p\mapsto \exp_{\eta(p)}(U(p)),$$
where $\exp$ is the Riemannian exponential map on $M$ (which takes the velocity vector $U(p)$ to the point on the geodesic through $\eta(p)$ in direction $U(p)$ at time one). In the Riemannian metric 
\begin{equation}\label{noninvariantmetric}
\llangle U,V\rrangle_{\eta} = \int_N \langle U(p), V(p)\rangle_{\eta(p)} \, d\nu,
\end{equation}
$\Exp$ is precisely the Riemannian exponential map. On the space $C^{\infty}(N,M)$, the map $\Exp$ is globally defined, although an obvious problem is that physically maps from $N$ to $M$ should be one-to-one: the physically natural space is the space of smooth embeddings $\Emb(N,M)$, but the exponential map is no longer globally defined on this space since it's very easy for geodesics to intersect. See \cite{ebinmarsden} for details. We note that the metric \eqref{noninvariantmetric} is invariant under the right-action of $\Diff_{\nu}(N)$ and under the left-action of the isometry group of $M$. Hence even if $M=N$, the metric is neither right- nor left-invariant on the open subset $\Diff(M)$.

If $\eta\in C^{\infty}(N,M)$ is a diffeomorphism, then we can express any element $U\in T_{\eta}C^{\infty}(N,M)$ as $U = u\circ\eta$, where $u$ is a vector field on $M$. By the change of variables formula for integrals, the Riemannian metric \eqref{noninvariantmetric} becomes the more familiar expression
$$ \llangle u\circ\eta, v\circ\eta\rrangle_{\eta} = \int_M \rho(q) \langle u(q), v(q)\rangle \, d\mu.$$
The geodesic equation $\frac{D}{dt} \frac{d\eta}{dt} = 0$ on $C^{\infty}(N,M)$ becomes, in terms of $u$ defined by 
\begin{equation}\label{flowequation}
\eta_t(t,p) = u\big(t, \eta(t,p)\big),
\end{equation} the Burgers' equation (or pressureless Euler equation) 
$u_t + \nabla_uu = 0$, while differentiating the density formula \eqref{densitydef} leads to the continuity equation \eqref{continuity}.

To obtain the barotropic equation \eqref{evolution}, we define (see Ebin~\cite{ebin} or Smolentsev~\cite{smolentsev}) a potential energy function $\Psi\colon C^{\infty}(N, M)\to \mathbb{R}$ which depends only on the density, of the form 
\begin{equation}\label{potential} 
\Psi(\eta) = \int_M \rho \psi(\rho) \, d\mu = \int_N \psi\left(\frac{1}{\Jac(\eta)}\right) \, d\nu.
\end{equation}
We can then compute that the gradient of this function in the metric \eqref{noninvariantmetric} is 
$$ (\grad \Phi)_{\eta} = \left( \frac{1}{\rho} \grad p(\rho)\right) \circ \eta, \qquad \text{where}\quad p(\rho) = \rho^2 \psi'(\rho),$$
and hence Newton's equation $\frac{D}{dt}\frac{d\eta}{dt} = -(\grad \Phi)_{\eta}$ can be written in the form \eqref{evolution} together with \eqref{flowequation}.
As pointed out by Smolentsev~\cite{smolentsev}, this system (like any conservative Newtonian system) can be rewritten as a geodesic equation with a modified metric---the Jacobi or Maupertuis metric. Generally speaking, if $\gamma$ satisfies Newton's equation $\frac{D}{dt}\frac{d\gamma}{dt} = -(\grad \Phi)_{\gamma(t)}$ on some manifold, then the energy $E=\frac{1}{2} \lvert \dot{\gamma}(t)\rvert^2 + \Phi(\gamma(t))$ is constant in time, and all solutions with the same energy will be geodesics in the conformally equivalent Riemannian metric 
$$ (u, v)_{\gamma} = \big(E-\Phi(\gamma)\big) \langle u, v\rangle_{\gamma}.$$ In this paper we will take a different approach which leads to a different formula for the Riemannian curvature.

In recent years many authors have studied infinite-dimensional geodesic equations which arise on groups of diffeomorphisms or related groups with right-invariant metrics. The geodesic equation in this case takes the form $\dot{\eta}(t) = dR_{\eta(t)} u(t)$ and $\dot{u}(t) + \ad^{\star}_{u(t)}u(t) = 0$, and the second equation is the Euler-Arnold equation. See \cite{arnoldkhesin} for examples. The closest relevant example in the present situation is the semidirect product $\Diff(S^1)\ltimes C^{\infty}(S^1)$. The Euler-Arnold equation was computed in \cite{guha} to be 
$$ u_t + 3uu_x + ff_x = 0, \qquad f_t + fu_x + uf_x = 0.$$
Although this resembles the system \eqref{evolution}--\eqref{continuity}, the extra factor of $3$ makes it genuinely different, and this factor cannot be removed by a rescaling without changing the flow equation $\eta_t = u\circ\eta$. In addition, the higher-dimensional version as in \cite{vizman} does not even superficially resemble the equation \eqref{evolution}. More general Euler-Arnold equations on semidirect product groups were considered in \cite{HMR}, who also considered an alternative Lagrangian approach to the barotropic equations (where $\rho$ is considered as an independent variable rather than derived from flow $\eta$). Our approach uses a non-invariant metric which does not lead to an Euler-Arnold equation.

\section{The geometry of $C^{\infty}(N, M\times \mathbb{R})$}\label{semidirectgeometry}

Now we define our (noninvariant) Riemannian metric. Our configuration space will be $C^{\infty}(N, M)\times C^{\infty}(N, \mathbb{R}) = C^{\infty}(N, M\times \mathbb{R})$. Tangent vectors to a point $(\eta, F)$ are of the form $(U,\phi)$ where $U\in T_{\eta}C^{\infty}(N,M)$ and $\phi \in C^{\infty}(N,\mathbb{R})$. If $\eta$ is a diffeomorphism, we can express $(U,\phi) = (u\circ\eta, f\circ\eta)$, where $u$ is a vector field on $M$ and $f$ is a function on $M$. 
Let $\lambda\colon \mathbb{R}^+\to \mathbb{R}^+$ be some smooth function, and define a metric on the product $C^{\infty}(N, M \times \mathbb{R})$ by the formula 
\begin{equation}\label{Semimetric}
\llangle (u\circ\eta, f\circ\eta), (v\circ\eta, g\circ\eta)\rrangle_{(\eta, F)} = 
\int_M \big[ \lambda(\rho) fg + \rho\langle u, v\rangle \big] \, d\mu.
\end{equation}
This is essentially a warped product of $C^{\infty}(N, M)$ in the noninvariant metric \eqref{noninvariantmetric} with the space of functions on $N$.

\subsection{Basic formulas}

The following Lemma shows us how to differentiate functions on our configuration space. The proof is a computation which can be found in \cite{smolentsev} or \cite{ebin}.

\begin{lemma}\label{functiondiff}
Suppose $W$ is a vector field on $C^{\infty}(N, M\times \mathbb{R})$ given by $W_{(\eta, F)} = (w\circ\eta, h\circ\eta)$ for some vector field $w$ on $M$ and function $h\colon M\to \mathbb{R}$. Suppose $\Phi\colon C^{\infty}(N, M\times \mathbb{R})\to \mathbb{R}$ is a function of the form 
$$ \Phi(\eta, F) = \int_M \alpha \varphi(\rho) \, d\mu = \int_N \alpha\circ\eta \, \Jac(\eta) \, \varphi(1/\Jac(\eta)) \, d\nu$$
for some functions $\alpha\colon M\to \mathbb{R}$ and $\varphi\colon \mathbb{R}^+\to\mathbb{R}$, where $\rho$ is defined by \eqref{densitydef}. Then the derivative of $\Phi$ in the direction $W$ is given at any point $(\eta, F)$ by 
\begin{equation}\label{densityfunctionderivative}
 W_{(\eta, F)}\Phi = \int_M \rho \langle w, \grad\big(\alpha \varphi'(\rho)\big)\rangle \, d\mu = -\int_M \diver{(\rho w)} \alpha \varphi'(\rho) \, d\mu.
 \end{equation}
\end{lemma}

Now that we know how to differentiate functions that depend only on the density, we can use the Koszul formula to obtain the covariant derivative. For our purposes it is sufficient to compute in terms of vector fields of the form $W_{(\eta, F)} = (w\circ\eta, h\circ\eta)$ for some vector field $w$ on $M$ and function $h\colon M\to \mathbb{R}$, as in Lemma \ref{functiondiff}. 

\begin{lemma}\label{covariantderivative}
Let $u$ and $v$ be vector fields on $M$, and let $f$ and $g$ be real-valued functions on $M$. Define vector fields $U$ and $V$ on $C^{\infty}(N, M\times \mathbb{R})$ by $U_{(\eta, F)} = (u\circ\eta, f\circ\eta)$ and $V_{(\eta, F)} = (v\circ\eta, g\circ\eta)$. Then the covariant derivative $\nabla_UV$ is given by 
$ (\nabla_UV)_{(\eta, F)} = Z\circ\eta$, where 
$$ Z = \big( u(g), \nabla_uv\big) + \Gamma_{\rho}\big( (u,f), (v,g)\big) $$
and $$\Gamma_{\rho}\colon \big(\chi(M)\times C^{\infty}(M,\mathbb{R})\big)^2 \to \chi(M)\times C^{\infty}(M, \mathbb{R})$$ is the Christoffel map, a bilinear map satisfying
\begin{equation}\label{christoffelproduct}
\llangle \Gamma_{\rho}\big((u,f), (v,g)\big), (w,h)\rrangle = \int_M \varphi(\rho) \big( hf \diver{v} + hg \diver{u} - fg \diver{w}\big) \, d\mu,
\end{equation}
where $\rho$ is given by \eqref{densitydef} and
\begin{equation}\label{phidef}
\varphi(\rho) = \frac{1}{2} \big( \lambda(\rho) - \rho \lambda'(\rho)\big).
\end{equation}
\end{lemma}

\begin{proof}
This is straightforward using the Koszul formula (see e.g., \cite{docarmo}), Lemma \ref{functiondiff}, and the fact that the commutator of right-invariant fields is $[U,V]_{(\eta, F)} = \big( [u,v]\circ\eta, (u(g)-v(f))\circ\eta\big).$
\end{proof}

\begin{remark} Note that in general a weak Riemannian metric does not necessarily have any covariant derivative at all; what happens here is essentially the same sort of fortunate accident that occurs on the full diffeomorphism group $\Diff(M)$ as in Ebin-Marsden~\cite{ebinmarsden}. Here we are effectively working on $C^{\infty}(N, M\times \mathbb{R})$, and using the fact that the covariant derivative constructed by \cite{ebinmarsden} on $\Diff(M)$ works in the same way on $C^{\infty}(N,X)$ for any Riemannian manifold $X$.  Of course the usual covariant derivative on $C^{\infty}(N,M\times \mathbb{R})$ is somewhat different since that comes from a Cartesian product metric and our metric is a warped product; however the difference of two connections is an operator defined pointwise (the Christoffel symbol), and so we get our connection as long as we understand this difference. Lemma \ref{covariantderivative} basically just computes this difference. 
\end{remark}

Although the formula \eqref{christoffelproduct} will be most convenient for our purposes, it is easy to see via integration by parts that we can write $\Gamma_{\rho}$ more explicitly as 
\begin{multline}\label{gammaexplicit}
\Gamma_{\rho}\big( (u\circ\eta,f\circ\eta), (v\circ\eta,g\circ\eta) \big) = (z\circ\eta, j\circ\eta) \qquad \text{where}  \\
j = \frac{\varphi(\rho)}{\lambda(\rho)} (f\diver{v} + g\diver{u})\quad\text{and}\quad z = \frac{1}{\rho} \, \grad\big( \varphi(\rho) fg \big).
\end{multline}
Notice that if $\lambda(\rho)=\rho$, then $\varphi(\rho)=0$ and hence $\Gamma_{\rho}=0$ as well. In this case the geometry reduces to the geometry of $M$ (as we will see later), and the covariant derivative is 
\begin{multline}\label{flatcovderiv}
(\widetilde{\nabla}_UV)_{(\eta, F)} = DR_{(\eta, F)} \big( u(g), \nabla_uv\big) = \big( u(g)\circ\eta, \nabla_uv\circ\eta\big) \\
\text{when $U_{(\eta, F)} = (u\circ\eta, f\circ\eta)$ and $V_{(\eta, F)} = (v\circ\eta, g\circ\eta)$}.
\end{multline}
This will be useful later.
Having obtained the covariant derivative, we can start computing. Our first goal is to obtain the geodesic equation.

\begin{corollary}\label{geodesiccorollary}
Suppose $\big(\eta(t), F(t)\big)$ is a geodesic curve in $C^{\infty}(N, M\times\mathbb{R})$ in the Riemannian metric \eqref{Semimetric}, with $\eta(t)$ a diffeomorphism from $N$ to $M$ on some interval $(-T, T)$. Define $\rho(t)$ by \eqref{densitydef}. Let $u(t)$ be the vector field satisfying $\frac{\partial \eta}{\partial t} = u\circ\eta$, and let $f(t)$ be the function satisfying $\frac{\partial F}{\partial t} = f\circ\eta$. Set $q = \lambda(\rho) f$. Then 
$u$ and $q$ satisfy the equations
\begin{align}
u_t + \nabla_uu + \frac{1}{\rho} \, \grad \Big( \frac{q^2 \varphi(\rho)}{\lambda(\rho)^2}\Big) &= 0 \label{evolutionuq} \\
q_t + \diver{(qu)} &= 0, \label{qcontinuity}
\end{align}
where $\varphi$ is given by \eqref{phidef}.
\end{corollary}

\begin{proof}
If we define $u$ and $f$ by $\dot{\eta}(t) = u(t)\circ\eta(t)$ and $\dot{F}(t) = f(t)\circ\eta(t)$ and $U(t) = \big(u(t), f(t)\big)$, then the geodesic equation can be written as $\dot{U}(t)\circ\eta(t) + (\nabla_{U(t)}U(t))\circ\eta(t) = 0$. 
Using formula \eqref{gammaexplicit}, this becomes
\begin{align*}
\frac{\partial u}{\partial t} + \nabla_uu + \frac{1}{\rho} \, \grad\big( \varphi(\rho) f^2\big) &= 0 \\
f_t + u(f) + \frac{2\varphi(\rho)}{\lambda(\rho)} \, f \diver{u} &= 0.
\end{align*}
Using \eqref{continuity}, it is then easy to see that $q=\lambda(\rho)f$ satisfies \eqref{qcontinuity}.
\end{proof}

Obviously \eqref{qcontinuity} is exactly the same differential equation as \eqref{continuity}, which implies that if $q=\rho$ at time zero, then $q=\rho$ for all time. In this case, the equation for $u$ satisfies 
$$ u_t + \nabla_uu + \frac{1}{\rho} \, \grad\Big( \frac{\rho^2 \varphi(\rho)}{\lambda(\rho)^2}\Big) = 0.$$ 
This is precisely equation \eqref{evolution} if we define 
\begin{equation}\label{pressurelambda}
p(\rho) = \frac{\rho^2 \varphi(\rho)}{\lambda(\rho)^2}.
\end{equation}
We have thus obtained the barotropic equations as a special case of the geodesic equation \eqref{evolutionuq}--\eqref{qcontinuity}. We summarize this as another corollary.

\begin{corollary}\label{compressiblegeodesic}
Suppose $\gamma_0 = (\eta_0, F_0)$ is a point in $\SemiNM$, and set $\rho_0$ to be the density of $\eta_0$ given by \eqref{densitydef}. Let $U_0\in T_{(\eta_0, F_0)}\SemiNM$ be the vector $U_0= (u_0\circ\eta_0, \frac{\rho_0}{\lambda(\rho_0)}\circ\eta_0)$. Then the geodesic $\gamma(t)$ with initial position $\gamma_0$ and initial velocity $U_0$ has tangent vector $\gamma'(t) = DR_{\gamma(t)}\big(u(t), f(t)\big)$, where $u$ satisfies the barotropic evolution equation \eqref{evolution} and $f = \frac{\rho}{\lambda(\rho)}$ for all time. 
\end{corollary}

We thus obtain the equations of interest if we restrict the initial velocity in $\SemiNM$ to be something very specific in the function direction (but arbitrary in the diffeomorphism direction). 
This gives a distribution of allowable velocities in $\SemiNM$. It is easy to see that this distribution is \emph{nonholonomic}; in other words, there is no submanifold of $\SemiNM$ for which all geodesics will correspond to barotropic flow. This limits the applicability of our standard geometric techniques, although we could if desired make sense of this situation in the context of the \emph{symmetric product}; see for example \cite{BL} and references therein for the general context.

\begin{remark}\label{entropyremark}
It is natural to ask whether the geodesic equations \eqref{evolutionuq}--\eqref{qcontinuity} have any meaning if our initial condition is not of the special type in Corollary \eqref{compressiblegeodesic}. An interpretation is as follows.
More general compressible fluid mechanics allows for the pressure $p$ in \eqref{evolution} to depend on both the density $\rho$ and the entropy $s$. In the absence of shocks, the entropy of any fluid particle is conserved, which implies the equation 
\begin{equation}\label{entropy}
s_t + u(s) = 0.
\end{equation}
If instead of having $q=\rho$ in \eqref{qcontinuity} we have $q=\rho \zeta(s)$ for some function $\zeta\colon \mathbb{R}\to\mathbb{R}$, then the fact that $\rho_t + \diver{(\rho u)}=0$ implies that $s$ must satisfy \eqref{entropy}, as expected. The corresponding pressure function would then be 
\begin{equation}\label{pressureentropy}
p(\rho, s) = \frac{\rho^2\varphi(\rho)}{\lambda(\rho)^2} \zeta(s)^2.
\end{equation}
In other words, we can represent any compressible fluid as a certain family of geodesics in $C^{\infty}(N, M\times \mathbb{R})$ as long as the pressure is separable as a function of its arguments, and conversely every geodesic in $C^{\infty}(N, M\times \mathbb{R})$ is a compressible fluid flow for some choice of the separable pressure function.
\end{remark}

\subsection{The sign of the curvature}\label{curvaturesignsection}

In spite of Remark \ref{entropyremark} we will still consider in the rest of this paper only barotropic flow. Since our primary motivation for considering the barotropic equations as geodesics is to understand stability in terms of curvature, we will compute the sectional curvature of this manifold. We are interested primarily in those sections where at least one of the vectors lies in our nonholonomic distribution, but we will first work out the general formula.

\begin{theorem}\label{curvaturethm}
Let $U$ and $V$ be right-invariant vector fields on the manifold $\SemiNM$ given by $U_{(\eta, F)} = (u\circ\eta, f\circ\eta)$ and $V_{(\eta, F)} = (v\circ\eta, g\circ\eta)$ for some vector fields $u$ and $v$ on $M$ and functions $f$ and $g$ on $M$. Then the (unnormalized) sectional curvature of the metric \eqref{Semimetric} is given by 
\begin{equation}\label{curvatureformula}
\begin{split}
&\llangle R(U,V)V, U\rrangle_{(\eta, F)} = \int_M \rho \langle R(u,v)v,u\rangle \, d\mu \\
&\qquad\qquad + \int_M \left( \rho \varphi'(\rho) + \frac{\varphi(\rho)^2}{\lambda(\rho)} \right) \, \big[ f \diver{v} - g\diver{u}\big]^2 \, d\mu \\
&\qquad\qquad + \int_M \varphi(\rho) \big[ f^2 Q(v,v)+g^2 Q(u,u) - 2fg Q(u,v)\big] \, d\mu \\
&\qquad\qquad + \int_M \frac{\varphi(\rho)^2}{\rho} \lvert f\grad g - g\grad f\rvert^2 \, d\mu,
\end{split}
\end{equation}
where the symmetric operator $Q$ is defined by 
\begin{equation}\label{Qdef}
Q(u,v) = \diver{(\nabla_uv)} - u(\diver{v}) - (\diver{u}) (\diver{v})
\end{equation} 
and $\langle R(u,v)v,u\rangle$ is the curvature on $M$.
\end{theorem}

\begin{proof}
This is a long but straightforward computation using Lemma \ref{covariantderivative}. 
The fact that $Q$ is symmetric follows from the computation
$$ Q(u,v)-Q(v,u) = \diver{[u,v]} - u(\diver{v}) + v(\diver{u}) = 0.$$ 
\end{proof}

The terms of \eqref{curvatureformula} are for the most part relatively simple, and thus we can determine the sign of the curvature quite easily, at least in special cases.

\begin{corollary}\label{1dcurvaturecorollary}
If $N=M=S^1$ and if 
\begin{equation}\label{nonnegativity}
x\varphi'(x)+\frac{\varphi(x)^2}{\lambda(x)} \ge 0 \quad \text{for all $x\ge 0$},
\end{equation} 
then the sectional curvature \eqref{curvatureformula} of $\SemiNM$ is nonnegative in all sections. 
\end{corollary}

\begin{proof}
In one dimension the first and third terms of \eqref{curvatureformula} always vanish, so positivity of the second term is sufficient.
\end{proof}

Note that we cannot prove strict positivity: if $f$, $g$, $u$, and $v$ are $C^{\infty}$ functions with disjoint supports in $S^1$, we will have $\llangle R(U,V)V,U\rrangle=0$.

\begin{remark}\label{pressurecurvature}
For a \emph{polytropic} fluid, where $p(\rho) = A\rho^{\gamma}$ for some constants $A$ and $\gamma>1$, it is easy to check that the condition \eqref{nonnegativity} is equivalent to $\gamma\le 3$. 
This is true in many applications (e.g., for typical gases at room temperature, $\gamma=1.4$; later we will consider the special cases $\gamma=3$ in one dimension and $\gamma=2$ in two dimensions). 
\end{remark}

In higher dimensions the formula gets more complicated, and ensuring positivity of curvature is more difficult. 
We choose two examples for which it is easy to compute the Jacobi fields (and hence completely determine the linearized stability) in order to illustrate the relationship between positive curvature and stability. We emphasize that since the Riemannian exponential map is not smooth or even $C^1$, we cannot rigorously prove any relationship between positive curvature and boundedness of Jacobi fields using tools like the Rauch theorem, even for short time. Hence these computations should be viewed as intuitive guides rather than directly useful for stability analysis.

\begin{exmp}\label{uniformprop}
Suppose $M=\mathbb{T}^2$ with flat Riemannian metric and $u$ is a velocity field of the form $u=\omega(x) \, \frac{\partial}{\partial y}$. Then for any pressure function $p(\rho)$, the velocity field $u$ is a solution of the steady compressible Euler equations 
\begin{equation}\label{steadyeuler}
\nabla_uu = -\frac{1}{\rho} \grad p(\rho), \qquad \diver{(\rho u)} = 0
\end{equation}
with constant density $\rho\equiv 1$. 
Writing $U=\big(u, \rho/\lambda(\rho)\big)$ and $V=\big(v, \rho/\lambda(\rho)\big)$, the curvature is given by
\begin{equation}\label{uniformcurvature}
\llangle R(U, V)V, U\rrangle = \frac{ \varphi'(1) + \varphi(1)^2/\lambda(1)}{\lambda(1)^2} \, \int_{\mathbb{T}^2} (\diver{v})^2 \, dx\,dy.
\end{equation}

This works because $\nabla_uu=0$ for such a velocity field, which forces $\rho$ to be constant. Since $Q(z,z)$ integrates to zero for any velocity field $z$, the corresponding terms disappear.
\end{exmp}

\begin{exmp}\label{rotational}
Suppose $M$ is the unit disc in $\mathbb{R}^2$. Let $u = \frac{\partial}{\partial \theta}$ be a rotational velocity field. Suppose that the pressure function is given by $p(\rho) = c^2\rho^2/2$ for some constant $c$, so that $\lambda(\rho) = \frac{1}{c^2}$ and $\varphi(\rho) = \frac{1}{2c^2}$. 

Then $u$ is a solution of the steady Euler equation \eqref{steadyeuler} with density function $\rho(r) = \frac{r^2}{2c^2}$, and the sectional curvature in directions $U=\big(u, \rho/\lambda(\rho)\big)$ and $V=\big(ku, \rho/\lambda(\rho)\big)$ for some constant $k$ is given by 
$$ \llangle R(U,V)V,U\rrangle = -\frac{\pi (k-1)^2}{48 c^2}.$$ To see this, just notice that since $f=g=\rho/\lambda(\rho)$ in \eqref{curvatureformula}, it reduces to 
$$ \llangle R(U,V)V,U\rrangle = \frac{c^2}{4} \int_M \rho^2 (\diver{z})^2 \, d\mu + \frac{c^2}{2} \int_M \rho^2 Q(z,z) \, d\mu,$$
where $z=v-u$. Now $\diver{z}=0$ while $Q(z,z) = -2(k-1)^2$, and the rest is an easy computation.
\end{exmp}

\begin{remark}
The geometric approach of Smolentsev~\cite{smolentsev}, using the Jacobi metric, yields an alternative sectional curvature formula; let us compute its sign. For simplicity we will work in the one-dimensional case. 
In this case the configuration space is $\Diff(S^1)$, and it is natural to compute the sectional curvature at an $\eta\in\Diff(S^1)$ in a plane spanned by $U=u\circ\eta$ and $V=v\circ\eta$. We can suppose the fields are normalized so that $\llangle U,V\rrangle = \int_{S^1} \rho(x) u(x) v(x) \, dx = 0$, and $\llangle U,U\rrangle=\llangle V,V\rrangle=1$. After some simplifications, the formula of Smolentsev then yields that the sectional curvature $\tilde{K}$ of $\Diff(S^1)$ in the Jacobi metric is given by 
\begin{multline*} 
\tilde{K} = \frac{1}{4(E-\Phi(\eta))^2} \bigg[ 
2 \int_{S^1} \rho p'(\rho) (u'^2 + v'^2) \, dx \\
+ \frac{3\big(\int_{S^1} p'(\rho) \rho' v\,dx\big)^2 
+ 3\big(\int_{S^1} p'(\rho) \rho' u \, dx\big)^2 - \int_{S^1} p'(\rho)^2 \rho'^2/\rho \, dx}{E-\Phi(\eta)}\bigg].
\end{multline*}
Clearly the fact that the first term involves $(u'^2+v'^2)$ implies that for typical $u$ and $v$ we obtain positivity. 

But we now ask whether this curvature can ever be negative. Now the parameter $E$ is arbitrary, as long as $E>\Phi(\eta_0)$, and thus the term $1/[E-\Phi(\eta)]$ can be made arbitrarily large. In addition, for a given $\rho$ and $p$, we can certainly choose $u$ and $v$ so that $\int_{S^1} p'(\rho) \rho' v\, dx = \int_{S^1} p'(\rho) \rho' u \, dx = 0$. Hence for any nonconstant density $\rho$, for values of $E$ sufficiently close to $\Phi(\eta)$ there will be velocity fields $u$ and $v$ such that the sectional curvature is approximately
$$ \tilde{K} \approx -\frac{1}{4(E-\Phi(\eta))^3} \int_{S^1} p'(\rho)^2 \rho'^2/\rho \, dx < 0.$$

We mention this only to emphasize that the Jacobi-metric approach is fundamentally different from our approach as far as curvature and stability predictions go, in spite of the fact that both generate the same geodesic equations. 
\end{remark}

\section{Jacobi fields and linear stability of compressible motion}\label{jacobisection}

In this section we compute some examples of Jacobi fields explicitly, partly to illuminate the Lagrangian stability theory and partly to discuss the properties of the Riemannian exponential map. We work with the case $p(\rho) = \frac{1}{3} \rho^3$ on the circle (where the Euler equations \eqref{evolution}--\eqref{continuity} reduce to the Burgers equation and can be solved fairly explicitly) and the case $p(\rho) = \frac{c^2}{2} \rho^2$ in two dimensions.

Note that although we have formulas for the covariant derivative \eqref{gammaexplicit} and the curvature \eqref{curvatureformula}, which in principle allow us to study the Jacobi equation directly, the formulas are complicated enough and there is enough cancellation in them to make them more trouble than they are worth. Instead we work directly with the linearizations of equations \eqref{evolution}--\eqref{continuity}.

\begin{theorem}\label{linearizations}
Consider a family $\zeta(s,t)$ of geodesics in $\SemiNM$, depending on a small parameter $s$, with tangent vectors satisfying the condition of Corollary \ref{compressiblegeodesic}, with $\zeta(0,t)=(\eta(t), F(t))$ and $\zeta(s,0)=\id$ for all $s$. Let $u$ and $\rho$ be the solutions of \eqref{evolution}--\eqref{continuity}, related to $\eta$ through \eqref{densitydef} and \eqref{flowequation}. 

Let $J(t) = \frac{\partial \zeta}{\partial s}(0, t)$ be the corresponding Jacobi field. Then $J(t) = (j(t)\circ\eta(t), G(t))$ for some vector field $j(t)$ and function $G(t)$, 
which satisfy $\sigma = -\diver{(\rho j)}$ and $\dot{G(t)} = g(t)\circ\eta(t)$ along with the linearized Lagrangian equations 
\begin{equation}\label{linearizedtransportflow}
g = 2\varphi(\rho)\sigma/\lambda(\rho)^2 + j\rho/\lambda(\rho) \quad \text{and} \quad
\frac{\partial j}{\partial t} + [u, j] = v 
\end{equation}
and the linearized Euler equations 
\begin{align}
\frac{\partial \sigma}{\partial t} + \diver{(\sigma u)} + \diver{(\rho v)} &= 0 \label{linearizedcontinuity} \\
\frac{\partial v}{\partial t} + \nabla_uv + \nabla_vu +\grad\big(h'(\rho)\sigma\big) &= 0,\label{linearizedevolution}
\end{align}
where $h$ is the function defined by $h'(\rho) = p'(\rho)/\rho$. 
\end{theorem}

\begin{proof}
Write $\zeta(s,t) = \big(\overline{\eta}(s,t), \overline{F}(s,t)\big)$, with $\overline{u}(s,t)$ and $\overline{\rho}(s,t)$ such that 
$ \overline{\eta}_t = \overline{u}\circ\overline{\eta}$ and $\overline{\rho}\circ\overline{\eta} J(\overline{\eta}) = 1.$ Define 
$ \sigma = \partial_s\big|_{s=0} \overline{\rho}$ and  $v = \partial_s\big|_{s=0} \overline{u}$. Differentiating \eqref{evolution} and \eqref{continuity} with respect to $s$ 
yields \eqref{linearizedcontinuity} and \eqref{linearizedevolution}. 
Differentiating 
$\partial_t \overline{F} = \big(\overline{\rho}/\lambda(\overline{\rho})\big)\circ\overline{\eta}$ 
with respect to $s$ gives 
$$ \frac{\partial G}{\partial t} = \frac{\partial \overline{f}}{\partial s}\circ\overline{\eta} \big|_{s=0} + \Big\langle \grad \overline{f}\circ\overline{\eta}, \frac{\partial \overline{\eta}}{\partial s}\Big\rangle \big|_{s=0},$$
which reduces to the first part of \eqref{linearizedtransportflow}. 
The second part of \eqref{linearizedtransportflow} comes from differentiating the flow equation \eqref{flowequation} and using the general formula
$ \frac{D}{\partial s} \frac{\partial \eta}{\partial t} = \frac{D}{\partial t}\frac{\partial \eta}{\partial s}$ pointwise on $M$ (see e.g., do Carmo~\cite{docarmo}). Finally the relationship $\sigma = -\diver{(\rho j)}$ comes from differentiating equation \eqref{densitydef}.
\end{proof}

\subsection{The one-dimensional case}

Consider the one-dimensional case $M=S^1$ (corresponding to periodic motion on $\mathbb{R}$). 
We will assume $p(\rho) = \frac{1}{3} \rho^3$; this is one of the cases that can be solved explicitly (see Courant-Friedrichs~\cite{CF}, Chap. III,  Sect. 28), which allows us to write down all the Jacobi fields explicitly as well.

\begin{theorem}\label{1dstability}
Suppose $(\eta, F)$ is a compressible geodesic as in Corollary \ref{compressiblegeodesic} with density $\rho$ and velocity $u$ on $S^1$, with pressure function given by $p(\rho) = \rho^3/3$. Consider a Jacobi field along $(\eta, F)$ as in Theorem \ref{linearizations} with initial condition $J(0)=0$ and $J'(0) = (v_0, 0)$. 
Then as long as the solution exists we have
\begin{equation}\label{jacobitime}
\lVert j(t)\rVert_{L^{\infty}} \le t \lVert v_0\rVert_{L^{\infty}}.
\end{equation}
\end{theorem}

\begin{proof}
In one dimension with $p(\rho)=\rho^3/3$,  equations \eqref{evolution} and \eqref{continuity} become 
$u_t + uu_x + \rho \rho_x = 0$ and  
$\rho_t + u\rho_x + \rho u_x = 0.$ 
We thus find that the functions 
$\alpha_+ = u+\rho$ and $\alpha_-=u-\rho$
both satisfy Burgers' equation:
$\alpha_t + \alpha \alpha_x = 0$.
Let $\xi_{\pm}$ be the flow of $\alpha_{\pm}$; then of course we have $\partial_t^2 \xi_{\pm} = 0$, 
the solution of which is obviously
$\xi(t,x) = x + t\alpha_{0}(x).$ 
Hence $\alpha_{\pm}$ satisfy the implicit equations 
$ \alpha_{0}(x) = \alpha\big(t, x + t\alpha_0(x)\big).$ 
If  $\chi_{\pm}$ denotes the spatial inverse of $\xi_{\pm}$, 
then we have $\alpha_{\pm}(t,x) = \alpha_{0}\big(\chi_{\pm}(t,x)\big)$.

The Jacobi equations \eqref{linearizedtransportflow}--\eqref{linearizedevolution} in this case are
\begin{alignat}{3}
(\rho j)_x &= -\sigma, &\qquad j_t + uj_x - ju_x &= v, \label{1dj} \\
\sigma_t + (\sigma u + \rho v)_x &= 0, &\qquad v_t + (uv + \sigma \rho)_x &= 0. \label{1dsigv}
\end{alignat}
Define $\beta_+ = v+\sigma$ and $\beta_- = v-\sigma$.
Then the functions $\beta_{\pm}$ satisfy the linearized Burgers equation
$$ \beta_t + \alpha \beta_x + \alpha_x \beta = 0.$$
It is easy to see that the solution is
$$ \beta_{\pm}(t,x) = \frac{v_0\big(\chi(t,x)\big)}{\xi_x\big(t,\chi(t,x)\big)} = v_0\big(\chi(t,x)\big) \chi_x(t,x),$$ 
since $\sigma(0,x)=0$ implies $\beta_{\pm}(0,x)=v_0(x)$.
Solving for $\sigma$ and $v$, equations \eqref{1dj} imply
\begin{equation}\label{jacobisoln}
\rho(t,x) j(t,x) = \frac{1}{2} \int_{\chi_+(t,x)}^{\chi_-(t,x)} v_0(y)\,dy.
\end{equation}

Now since the inverse flows $\chi_{\pm}$ satisfy
$ \chi_{\pm}(t,x) = x - t \alpha_{\pm}(t,x)$, formula \eqref{jacobisoln} yields
\begin{align*} \rho(t,x) \lvert j(t,x)\rvert &\le \frac{1}{2} \sup_{y\in S^1} \lvert v_0(y)\rvert \lvert \chi_-(t,x)-\chi_+(t,x)\rvert \\
&\le
\frac{t}{2} \lVert v_0\rVert_{L^{\infty}} \lvert \alpha_+(t,x)-\alpha_-(t,x)\rvert \\
&= t\, \rho(t,x) \lVert v_0\rVert_{L^{\infty}},
\end{align*}
which implies \eqref{jacobitime}.
\end{proof}

Recall that the curvature is nonnegative by Corollary \ref{1dcurvaturecorollary} but not strictly positive. In fact although our $\gamma=3$ is the critical case in \eqref{nonnegativity}, which makes all terms but the last in \eqref{curvatureformula} vanish, the last term is generally positive, since we cannot expect the function $f=\rho/\lambda(\rho)$ to coincide with the function $g$ satisfying \eqref{linearizedtransportflow}. Hence the curvature is nonnegative but sometimes positive, and the linear growth of Jacobi fields aligns with our intuition and is the best we can expect. This is a sort of weak instability: polynomial but not exponential growth of the Lagrangian perturbations.

We can also use this explicit solution to establish the lack of smoothness of the Riemannian exponential map. Our technique involves analyzing the conjugate points along a particular geodesic, as in \cite{mkdv}, \cite{CK}, and \cite{prestonwhips}. The idea is that although the exponential map is \emph{continuous} in various function spaces (see Kato~\cite{kato}), it cannot be $C^1$ in any Banach space: if it were, then the inverse function theorem and the fact that its derivative at $0$ is the identity map would imply the existence of a small interval on the geodesic in which no two points are conjugate (see e.g., do Carmo~\cite{docarmo}). Hence we will prove the exponential map cannot be $C^1$ by finding a particular geodesic $\gamma$ such that $\gamma(t_n)$ is conjugate to $\gamma(0)$ for a sequence of times $t_n$ converging to $0$.

\begin{theorem}\label{notc1}
Let $H^s(S^1, S^1\times \mathbb{R})$ denote the closure of $C^{\infty}(S^1, S^1 \times \mathbb{R})$ in the Sobolev $H^s$ topology, and consider the weak Riemannian metric with $\lambda(\rho) = 3/\rho$. Then the Riemannian exponential map cannot be $C^1$ for any $s>3/2$. 
\end{theorem}

\begin{proof}
By \eqref{phidef} and \eqref{pressurelambda}, this choice of $\lambda$ gives $p(\rho)=\rho^3/3$ as in Theorem \ref{1dstability}. A particular solution of the equaitons 
is $\rho(t,x)=u(t,x)=1$ for all $t$ and $x$. Let $\gamma=(\eta, F)$ be the corresponding geodesic as in Corollary \ref{compressiblegeodesic}; to prove that $\gamma(0)$ and $\gamma(T)$ are conjugate, we need to find a Jacobi field with $J(0)=J(T)=0$. By theorem \ref{linearizations} we can write $J(t) = (j\circ\eta, G)$ where $j$ satisfies \eqref{1dj} and $\frac{\partial G}{\partial t} = g\circ\eta$, and we can compute that 
$g(t,x) = 
\frac{2}{3} \sigma(t,x).
$

For this solution we have (as in the proof of Theorem \ref{1dstability}) that $\alpha_+(t,x) = 2$ and $\alpha_-(t,x)=0$ for all $t$ and $x$, so that the inverse flows are $\chi_+(t,x) = x-2t$ and $\chi_-(t,x) = x$. 
Thus we obtain 
$$ \sigma(t,x) = \tfrac{1}{2} \big(v_0(x-2t) - v_0(x)\big) \quad \text{and}\quad v(t,x) = \tfrac{1}{2}\big( v_0(x-2t)+v_0(x)\big).$$ 
Formula \eqref{jacobisoln} yields 
$$ j(t,x) = \frac{1}{2} \int_{x-2t}^x v_0(y) \, dy$$ 
and since the flow of $u$ is $\eta(t,x) = x+t$, we obtain 
$$ G(t,x) = \int_0^t g\big(\tau, \eta(\tau, x)\big) \, d\tau = \frac{1}{3} \int_0^t \big(v_0(x-\tau) - v_0(x+\tau)\big) \, d\tau.$$

Now suppose $v_0(x) = \cos{nx}$ for some integer $n$. Then we have 
\begin{align*}
j(t,x) &= \frac{\sin{nt}\cos{n(x-t)}}{n} \\
G(t,x) &= \frac{4}{3n} \sin{(nx)} \sin^2{(nt/2)},
\end{align*}
and we see that there are conjugate points at $T=\frac{2\pi m}{n}$ for any positive integer $m$. 
Thus there is no neighborhood of $0$ on which the derivative of the exponential map is an invertible linear map.
\end{proof}

\subsection{The two-dimensional case}

We can do everything fairly explicitly in one dimension, at least assuming a simple pressure function. In two or more dimensions, explicit nonsteady solutions are much harder to come by. Hence we will just work out the Jacobi fields and Lagrangian stability for the two special steady solutions already studied in Section \ref{curvaturesignsection}, via Proposition \ref{uniformprop} and Example \ref{rotational}.

\begin{theorem}\label{uniformjacobi}
Suppose $M=\mathbb{T}^2$.
Suppose $u = \omega \, \frac{\partial}{\partial y}$, for some constant $\omega$, is a steady solution of the Euler equation \eqref{steadyeuler}, with constant uniform density $\rho\equiv 1$. Let $c$ be the speed of sound, defined by $c^2 = p'(1)$, for an arbitrary function $p(\rho)$. Let $\gamma=(\eta, F)$ be the geodesic as in Corollary \ref{compressiblegeodesic}. Let $J=(j\circ\eta, G)$ be a Jacobi field along $\gamma$ as in Theorem \ref{linearizations}, with initial conditions $j(0,x)=0$ and $j_t(0,x) = v_0(x)$ for some velocity field $v_0$. 
Then $j$ is bounded in time if and only if $v_0$ is a gradient; otherwise $j$ grows linearly in time at every point.
\end{theorem}

\begin{proof}
Equations \eqref{linearizedtransportflow}--\eqref{linearizedevolution} reduce in this case to 
$$ (\partial_t + \omega \partial_y)j = v, \quad (\partial_t + \omega \partial_y)\sigma + \diver{v} = 0, 
\quad (\partial_t + \omega \partial_y)v + c^2 \grad \sigma = 0.$$ 
Eliminating $\sigma$ we get $$(\partial_t+\omega \partial_y)^2 v = c^2 \grad \diver{v},$$
with initial condition satisfying
\begin{equation}\label{vinits}
v_t(0,x,y) + \omega v_y(0,x,y) = 0.
\end{equation}
It is easy to write down solutions: we expand $v$ using the Hodge decomposition as $v = \grad f + w$, 
where $f$ is some mean-zero function and $w$ is some divergence-free vector field which both depend on time. Then $f$ and $w$ satisfy the equations 
$$ (\partial_t + \omega \partial_y)^2 f = c^2 \Laplacian f \qquad \text{and}\qquad (\partial_t+\omega \partial_y)^2 w = 0.$$
The solution satisfying \eqref{vinits} is clearly
\begin{align*} 
f(t,x,y) &= \sum_n a_n \cos{(c \lambda_n t)}\phi_n(x, y-\omega t) \\
w(t,x,y) &= z(x, y-\omega t),
\end{align*}
where $\Laplacian \phi_n = -\lambda_n^2 \phi_n$, the numbers $a_n$ are some constants, and $z$ is some divergence-free field.
Integrating once to solve for $j$ we obtain
$$ j(t,x,y) = \sum_n \frac{a_n \sin{(c\lambda_n t)}}{c\lambda_n} \grad \phi_n(x,y-\omega t) + tz(x, y-\omega t).$$
Every component of the gradient part is bounded in time, while any nontrivial divergence-free part grows in time.
\end{proof}

Note that we obtain the behavior of Jacobi fields regardless of the curvature: in Proposition \ref{uniformprop} we showed the curvature was nonnegative if and only if $\varphi'(1) + \varphi(1)^2/\lambda(1)\ge 0$, which e.g., happens for pressure functions $p(\rho) = A\rho^{\gamma}$ if and only if $1<\gamma\le 3$. On the other hand even if this quantity is negative, it does not change the Jacobi fields in any way.

Now we consider our final example: a rigidly rotating disc in the plane, as in Example \ref{rotational}. We assume the steady solution takes the form $u = \omega \, \frac{\partial}{\partial \theta}$ for some constant $\omega$, and that the pressure function is given by $p(\rho) = c^2 \rho^2/2$, so that the Euler equation \eqref{evolution} is 
$ u_t + \nabla_uu = -c^2 \grad \rho$ for some constant $c$ (the speed of sound). In this case we must have $\rho'(r) = \omega^2 r/c^2$. Now the stability analysis depends on the boundary conditions we use. For our computations it is simpler to work with the free-boundary case (without surface tension), in which the density is specified on the boundary by $\rho=\rho_0$ for some constant. If instead we worked with a fixed boundary (e.g., the fluid in a solid container), then the density would be unspecified at the boundary but the velocity would be required to be tangent to the boundary.

We may assume by rescaling units that the boundary is at $r=1$ and that $\rho_0=1$. Then $\rho(1)=1$ and $\rho'(r)=\omega^2r/c^2$ imply that $\rho(r) = \rho_0 - \frac{\omega^2}{2c^2} + \frac{\omega^2 r^2}{2c^2}$. Since the density must always be positive, we must have $\rho_0 > \frac{\omega^2}{2c^2}$ for this to make sense. Hence the velocity may exceed the speed of sound but not more than by a factor of $\sqrt{2}$. Although in this supersonic region the evolution equations are not hyperbolic, and hence the standard existence theory breaks down, there is no apparent problem with the linearized equations in Theorem \ref{linearizations}.

Now we analyze the linear stability of the uniformly rotating fluid. Although stability can be analyzed for such fluids (and in much greater generality; see e.g., \cite{HMRW}), we are interested in obtaining the explicit time-dependence of solutions of the linearized equations in order to study growth of the Jacobi fields as in Theorem \ref{linearizations}. What is interesting about this computation is that we can express all the perturbations in terms of discrete Fourier modes: in general the fact that the corresponding operators are not self-adjoint means that we should expect a continuous spectrum as well as a discrete spectrum, but here the spectrum is purely discrete. The other reason this is interesting is that we don't require an ansatz for the growth of a perturbation; rather, the dependence is a consequence of the equations.

The only boundary condition for the linearized equations is that $\sigma=0$ on the boundary, a result of $\rho$ being prescribed. We note that Beyer and G\"unther~\cite{BG} analyzed the Jacobi equations in this situation, in greater generality but less detail.

\begin{theorem}\label{eulerianrigidrotation}
Suppose $M$ is the unit disc in $\mathbb{R}^2$, with $u=\omega \, \frac{\partial}{\partial \theta}$, a rigid rotation for some constant $\omega$. Further suppose that the barotropic fluid is described by the pressure function $p(\rho)=c^2\rho^2/2$ for some constant $c$. We assume the boundary conditions are such that the the density is a prescribed constant but the velocity is unconstrained.

Then all solutions of the linearized equations \eqref{linearizedcontinuity}--\eqref{linearizedevolution} are bounded in time, and we can write the solutions explicitly as sums of functions which look like $Q(t,r,\theta) = e^{iyt} q(r) e^{in(\theta-\omega t)}$ for some real $y$, some $n\in\mathbb{Z}$, and some function $q$.
\end{theorem}

\begin{proof}
The linearized equations \eqref{linearizedcontinuity}--\eqref{linearizedevolution} are
\begin{align}
\sigma_t + \diver{(\rho v + \sigma u)}&= 0, \label{stilllinearizedcontinuity} \\
v_t + \nabla_uv + \nabla_vu &= -c^2 \grad \sigma, \label{specificlinearizedevolution}
\end{align}
with boundary condition $\sigma=0$ when $r=1$.
To write these equations more explicitly, we decompose the vector field $\rho v$ into its gradient and divergence-free part (using the Hodge decomposition); more specifically 
\begin{equation}\label{vfg}
\rho v = \grad f  + \sgrad g = \frac{1}{\rho} \Big( \big( f_r + g_{\theta}/r\big) \, \partial_r + \big( f_{\theta}/r^2 - g_r/r\big) \, \partial_{\theta}\Big).
\end{equation}
Without loss of generality we can assume that $g$ vanishes on the boundary, although we cannot say anything yet about $f$ on the boundary.

We easily compute that
$$\nabla_vu = \frac{\omega}{r} \, \grad g - \frac{\omega}{r} \, \sgrad f.$$
Using the explicit formulas for $\rho$, $u$, and $v$, we can rewrite \eqref{stilllinearizedcontinuity} as 
\begin{equation}\label{reallyspecificcontinuity}
\sigma_t + \Laplacian f + \omega \sigma_{\theta} = 0.
\end{equation}
Taking the divergence and curl, we obtain 
\begin{align}
\sigma_t + \omega \sigma_{\theta} + F &= 0 \label{homo1} \\
F_t + \omega F_{\theta} + 2\omega G + c^2 \Lambda \sigma &= 0 \label{homo2} \\
G_t + \omega G_{\theta} - 2\omega F + \omega^2 \sigma_{\theta} &= 0, \label{homo3}
\end{align}
where $\Lambda \sigma = \diver{(\rho \grad \sigma)}$, $F = \Laplacian f$, and $G=\Laplacian g$.
By our free-boundary assumption, $\sigma$ vanishes on the boundary, and compatibility requires that $F$ and $G$ have the same boundary conditions.

Since the hermitian operator $\Lambda$ commutes with the antihermitian operator $\partial_{\theta}$, we can expand our functions in a mutual eigenbasis $\zeta_{kn}(r) e^{in\theta}$, where $\Lambda (\zeta_{kn}e^{in\theta}) = -\lambda_{kn} \zeta_{kn}(r)e^{in\theta}$, for some functions $\zeta_{kn}$ defined for $k\in\mathbb{N}$ and $n\in\mathbb{N}$. By the usual Sturm-Liouville theory, we see that $\lambda_{kn}\to\infty$ as $k\to\infty$ for any fixed $n$.

Then we can write 
$$ F(t,r,\theta) = \sum_{k=1}^{\infty} \sum_{n\in\mathbb{Z}} F_{kn}(t) \zeta_{kn}(r) e^{in\theta},$$
and similarly for $G$ and $\sigma$. Hence the coefficients satisfy the ordinary differential system 
\begin{equation}\label{timeonlylineuler}
\begin{split}
\frac{d\sigma_{kn}}{dt} + F_{kn}(t) &= 0, \\
\frac{dF_{kn}}{dt} + 2\omega G_{kn}(t) - c^2 \lambda_{kn} \sigma_{kn}(t) &= 0, \\
\frac{dG_{kn}}{dt} - 2\omega F_{kn}(t) + in \omega^2 \sigma_{kn}(t) &= 0.
\end{split}
\end{equation}

This is a constant-coefficient system, and its characteristic polynomial is $z^3+(c^2\lambda_{kn}+4\omega^2)z + 2in\omega^3=0$.  and we want to show that it has three distinct imaginary roots in order to guarantee that $\sigma$, $F$, and $G$ are all bounded in time. Writing $z=iy$, we obtain 
\begin{equation}\label{ycharpoly}
y^3 - 3py - 2q = 0, \quad \text{where $p=(c^2\lambda_{kn} +4\omega^2)/3$ and $q=n\omega^3$.} 
\end{equation}
Such a cubic has three distinct real roots if and only if $p>0$ and $q^2< p^3$.
So we want to show that $c^2\lambda_{kn} > \omega^2 (3n^{2/3}-4)$ for every $n\in \mathbb{Z}$ and every $k\in \mathbb{N}$. Obviously it is enough to show that $c^2\lambda_{1n}> \omega^2 (3n^{2/3}-4)$.

To obtain this bound on the smallest eigenvalue, 
we use the Rayleigh minimum principle. Let $\mathcal{F}$ be the space of smooth functions $f\colon [0,1]\to \mathbb{R}$ such that $f(1)=0$; then we can compute that
\begin{multline*} 
\lambda_{1n} \ge a \,\inf_{f\in \mathcal{F}} \frac{\int_0^1 r f'(r)^2 \, dr + n^2 \int_0^1 f(r)^2/r \, dr}{\int_0^1 rf(r)^2\,dr} \\
+ b\, \inf_{f\in \mathcal{F}} \frac{\int_0^1 r^3 f'(r)^2 \, dr + n^2 \int_0^1 r f(r)^2 \, dr}{\int_0^1 r f(r)^2\,dr},
\end{multline*}
where $a=\rho_0 - \frac{\omega^2}{2c^2}$ and $b=\frac{\omega^2}{2c^2}$. 
The first term is exactly the Rayleigh quotient for the Bessel operator, and hence it is minimized when $f(r) = J_n(c_n r)$ with $c_n$ the first positive root of the Bessel function $J_n$; the minimum value is then $c_n^2$. We claim the second quotient takes the minimum value $n^2+1$: to see this, write $h(r)=rf(r)$, so that $h(0)=h(1)=0$, and 
$$ 
\frac{\int_0^1 r^3 f'(r)^2 \, dr}{\int_0^1 r f(r)^2 \, dr} 
= 1 + \frac{\int_0^1 rh'(r)^2 \, dr}{\int_0^1 \frac{1}{r} h(r)^2 \, dr},
$$ 
and the infimum of this last term is zero, as can be seen by computing it for functions $h(r) = \frac{-\ln{r}}{1+\alpha(\ln{r})^2}$ for $\alpha>0$; we obtain $\frac{\alpha}{2}$, so we can make it as close as desired to zero. 
As a result we have $\lambda_{1n} \ge a c_n^2 + b(n^2+1) > \frac{\omega^2}{2c^2} (n^2+1)$. The fact that $c^2\lambda_{1n} > \omega^2 (3n^{2/3}-4)$ now follows from the fact that $n^2+7>6n^{2/3}$ for every integer $n$.

Hence all solutions of the system \eqref{timeonlylineuler} oscillate in time; specifically we can write $\sigma_{kn}(t) = \sigma_{1kn} e^{iy_{1kn}t} + \sigma_{2kn} e^{iy_{2kn}t} + \sigma_{3kn} e^{iy_{3kn}t}$ for some distinct reals $y_{ikn}$, and similarly for $F_{kn}$ and $G_{kn}$. Having obtained such a solution for $F$ and $G$, we can then find $f=\Laplacian^{-1}F$ and $g=\Laplacian^{-1}G$. Here $g$ is assumed to vanish on the boundary, and we can obtain the boundary values for $f$ using 
\eqref{stilllinearizedcontinuity}--\eqref{specificlinearizedevolution}.
\end{proof}

It is now trivial to figure out how Jacobi fields grow, using \eqref{linearizedtransportflow}.

\begin{corollary}
Let $\gamma=(\eta, F)$ be a geodesic in the space $C^{\infty}(D^2, D^2 \times \mathbb{R})$ for which the corresponding tangent vector satisfies the condition of Corollary \ref{compressiblegeodesic}, with velocity $u$ and density $\rho$ given as in Theorem \ref{eulerianrigidrotation}. Then Jacobi fields are bounded for all time if the initial perturbation $v_0$ satisfies $\int_0^{2\pi} \curl{\big(\rho(r) v_0(r,\theta)\big)} \, d\theta = 0$ for every $r$.
\end{corollary}

\begin{proof}
From \eqref{linearizedtransportflow} we see that $j_t + \omega j_{\theta} = v$, and we expressed $v$ as a sum of components of the form $v(t,r,\theta) = e^{iyt} e^{in(\theta-\omega t)} q(r)$ for some real number $y$. Hence 
the corresponding component of the Jacobi field is 
$$ j(t,r,\theta) = e^{in(\theta-\omega t)} q(r) \int_0^t e^{iys} \, ds.$$
This will be bounded for all time if and only if $y\ne 0$. From equation \eqref{ycharpoly}, we see that
$y=0$ iff $n=0$. It is easy to check that we have nonzero $n=0$ components if and only if the condition of the theorem is satisfied.
\end{proof}

This is another illustration of the fact that compressible flows tend to be more Lagrangian stable than incompressible flows. There are many Jacobi fields which are bounded in time in the compressible case, while in the incompressible case the curvature along a rigid rotational flow vanishes identically (it satisfies both the nonnegativity condition of Misio{\l}ek~\cite{misiolek} and the nonpositivity condition of the author~\cite{prestonnonpositive}) and thus all Jacobi fields grow linearly in time.

\makeatletter \renewcommand{\@biblabel}[1]{\hfill#1.}\makeatother

\end{document}